\documentclass[11pt,final]{amsart}
\usepackage[active]{srcltx}
\usepackage[latin1]{inputenc}
\usepackage{amsmath,bbm}
\usepackage{amsfonts}
\usepackage{pgf,tikz}
\usetikzlibrary{arrows}
\usepackage{amssymb}
\usepackage[colorlinks=true,urlcolor=blue,
citecolor=red,linkcolor=blue,linktocpage,pdfpagelabels,bookmarksnumbered,bookmarksopen]
{hyperref}
\usepackage{graphicx, tikz}
\usepackage{pgf,tikz}
\usetikzlibrary{arrows}
\usepackage[active]{srcltx}
\usepackage[latin1]{inputenc}
\usepackage{amsmath,bbm}
\usepackage{amsfonts}
\usepackage{amsthm}
\usepackage{amssymb}
\usepackage{amsfonts}
\usepackage{hyperref}
\usepackage{hyperref}
\usepackage{amsmath,accents,eqnarray}
\usepackage[english]{babel}
\usepackage{marginnote}
\usepackage[left=2.95cm,right=2.95cm,top=2.8cm,bottom=2.8cm]{geometry}
\numberwithin{equation}{section}
\usepackage{mathrsfs}
\usepackage{color}
\usepackage{amsmath, bbm}
\usepackage{amsfonts}
\usepackage{amssymb}
\usepackage{graphicx}%
\usepackage[active]{srcltx}
\usepackage[latin1]{inputenc}
\usepackage{amsmath,bbm}
\usepackage{amsfonts}
\usepackage{pgf,tikz}
\usetikzlibrary{arrows}
\usepackage{amssymb}
\newtheorem{theorem}{Theorem}[section]
\newtheorem{proposition}[theorem]{Proposition}
\newtheorem{lemma}[theorem]{Lemma}

\newtheorem{definition}{Definition}[section]
\newtheorem{remark}{Remark}[section]

\numberwithin{equation}{section}

\numberwithin{equation}{section}
\usepackage{mathrsfs}

\DeclareMathOperator{\spt}{spt}

\begin{document}
\title[$L^{p,q}$ estimates on the transport density]
{$L^{p,q}$ estimates on the transport density}
\author[S. Dweik]{Samer Dweik}
\address{Laboratoire de Math\'ematiques d'Orsay, Univ. Paris-Sud, CNRS, Universit\'e Paris-Saclay, 91405 Orsay Cedex, France}
\email{samer.dweik@math.u-psud.fr}
\maketitle
\begin{abstract} In this paper, we show a new regularity result on the transport density \,$\sigma$\, in the classical Monge-Kantorovich optimal mass transport problem between two measures, $\mu$ and \,$\nu$, having some summable densities, $f^+$ and \,$f^-$. More precisely, 
we 
 prove that the transport density \,$\sigma$\, belongs to \,$L^{p,q}(\Omega)$\, as soon as \,$f^+,\,f^- \in L^{p,q}(\Omega)$. 
\end{abstract}
\section{Introduction} Let \,$\mu$\, and \,$\nu$\, be two given non-negative Borel measures on a compact convex domain $\Omega \subset \mathbb{R}^d$, satisfying the mass balance condition $\mu(\Omega)=\nu(\Omega)$. Let $|\cdot|$ stand for the Euclidean norm in $\mathbb{R}^d$. The classical Monge problem (MP) \cite{Monge} consists of finding a {\it{transport map}}\, $T^*:\Omega \mapsto \Omega$\, minimizing the functional
$$ \int_{\Omega}|x-T(x)|\,\mathrm{d}\mu(x)$$\\
over all Borel measurable maps \,$T: \Omega \mapsto \Omega$\, satisfying \,$T_{\#}\mu=\nu$, where \,$T_{\#}$\, denotes the push-forward operator acting on every Borel measure \,$\mu$\, according to the formula \\
$$ T_{\#}\mu(B):=\mu(T^{-1}(B)) \;\,\,\mbox{for all Borel set}\,\, B\subset\Omega.$$
\\ This problem may have no solutions: this happens, for instance, when \,$\mu$\, is a Dirac mass and $\nu$\, is not. In \cite{Kanto},
 Kantorovich proposed a notion of weak solution to this transport problem. He suggested to look for {\it{transport plans}}\, instead of {\it{transport maps}}, i.e. non-negative measures \,$\gamma$\, on \,$\Omega \times \Omega$\, whose marginals are \,$\mu$\, and \,$\nu$. Formally, this means that $(\Pi_x)_{\#}\gamma=\mu$ and $(\Pi_y)_{\#}\gamma=\nu$, where $\Pi_x$ and $\Pi_y:\Omega \times \Omega \mapsto \Omega$ are the canonical projections. Denoting by $\Pi(\mu,\nu)$ the class of transport plans, he wrote the following minimization problem\\
$$\mbox{(KP)}\qquad \min\left\{\int_{\Omega \times \Omega}|x-y|\,\mathrm{d}\gamma:\;\gamma \in \Pi(\mu,\nu)\right\}.$$\\ 
Due to the convexity of the new constraint \,$\gamma \in \Pi(\mu,\nu)$\, and the linearity in \,$\gamma$\, of the functional, it turns out that weak topologies can be effectively used to provide existence of solutions to (KP). In fact, if $(\gamma_n)_n$ is a minimizing sequence, then, using Prokhorov's theorem, we have, up to a subsequence, $\gamma_n \rightharpoonup \gamma$ with $\gamma \in \Pi(\mu,\nu)$. Moreover, one has $\int_{\Omega \times \Omega}|x-y|\,\mathrm{d}\gamma_n \rightarrow \int_{\Omega \times \Omega}|x-y|\,\mathrm{d}\gamma$, which implies directly that $\gamma$ is optimal for (KP). \\
\\
The connection between the Kantorovich formulation of the transport problem and Monge's original one can be seen noticing that any transport map \,$T$\, induces a transport plan $\gamma$, defined by $(Id,T)_{\#} \mu$, which means that this plan is concentrated on the graph of \,$T$ in $\Omega \times \Omega$. We also see that the converse holds, i.e. whenever $\gamma$  is concentrated on a graph, then $\gamma$ is induced by a transport map. Since any transport map induces a transport plan with the same cost, it turns out that 
$$ \inf\, (\mbox{MP}) \geq \min \,(\mbox{KP}).$$
We also note that the equality $\inf\, (\mbox{MP})= \min \,(\mbox{KP})$ holds as soon as there is an optimal transport plan $\gamma$ which is concentrated on a graph $y=T(x)$. By the way, this map $T$ will be optimal for (MP). Yet, it has been really hard to give some answer about the existence of such an
optimal transport plan which is induced by a map.\\ \\
On the other hand, it is well known that the dual setting (DP) for the Monge-Kantorovich problem consists of finding a function $u$ (called {\it{Kantorovich potential}}) which maximizes the functional
$$ \int_{\Omega} v\,\mathrm{d}(\mu - \nu)$$\\
over all $v \in \mbox{Lip}_1(\Omega)$, where \,$\mbox{Lip}_1(\Omega)$\, stands for the set of Lipschitz continuous functions on $\Omega$ with Lipschitz constant one. This duality $\min\mbox{(KP)}=\sup\mbox{(DP)}$ implies that optimal $\gamma$ and $u$ satisfy
$$u(x) - u(y)=|x-y|\quad \mbox{on}\,\,\,\spt(\gamma).$$\\
We call {\it{transport ray}}\, any non-trivial (i.e., different from a singleton) segment $\left[x,y\right]$ such that $u(x) - u(y) = |x - y|$ that is maximal for the inclusion among segments of this form. Following this definition, we see that an optimal transport plan has to move the mass along the transport rays. And, it is well known that two different transport rays cannot intersect at an interior point of one of them (see, for instance, \cite{8}).\\ \\
Coming back to the problem of existence of optimal transport maps, Evans and Gangbo \cite{5} have made a remarkable progress showing by differential methods the existence of such a map, under the assumption that the two measures \,$\mu$\, and \,$\nu$\, are absolutely continuous with respect to $\mathcal{L}^d$, that their densities $f^+$ and $f^-$ are Lipschitz with compact supports and that \,$\spt(f^+) \cap \spt(f^-)=\emptyset$ (we note that after the work of Evans and Gangbo, Ambrosio in \cite{1} has proved that there exists an optimal transport map for the Monge problem provided that $\mu \ll \mathcal{L}^d$). A solution to the classical Monge-Kantorovich problem can be constructed by studying the $p-\mbox{Laplacian}$ equation 
$$ -\nabla \cdot(|\nabla u_p|^{p-2}\nabla u_p)=f^+-f^- $$
in the limit as $p \rightarrow +\infty$. They show that \,$u_p \rightarrow u$\, uniformly, where \,$u$\, is a Kantorovich potential between $f^+$ and $f^-$, and at the same time, they prove the existence of a special non-negative function \,$a$\, such that 
$$ -\nabla \cdot (a \nabla u)=f^+-f^-,\;|\nabla u|=1\;\;\,\mathcal{L}^d-\;\,\mbox{a.e. on}\,\; \{a>0\}.$$\\
The diffusion coefficient \,$a$\, in the PDE above plays a special role in the theory. Indeed, one can show that the measure \,$\sigma:=a \cdot \mathcal{L}^d$ (the so-called $\textit{transport density}$) can be represented in several different ways, and in particular as 
\begin{equation}\label{transport density def}
 \sigma(A)=\int_{\Omega \times \Omega} \mathcal{H}^1(A \cap [x,y])\,\mathrm{d}\gamma(x,y)\;\;\;\,\,\mbox{for all Borel set}\,\,\,A\subset \Omega,
 \end{equation} for some optimal transport plan $\gamma$, where \,$\mathcal{H}^1$ stands for the 1-dimensional Hausdorff measure. Its physical meaning is the work for transporting the mass through the set $A$. It has been proven in \cite{33,98} that if either $\mu$ or $\nu$ is absolutely continuous with respect to $\mathcal{L}^d$, then $\sigma$ is unique (i.e. does not depend on the choice of the optimal plan $\gamma$) and it is also absolutely continuous with respect to \,$\mathcal{L}^d$. Moreover, the authors of \cite{558,778,98} proved the following $L^p$ result on the transport density $\sigma$:
$$ f^+,\,f^- \in\;L^p(\Omega)\;\Rightarrow\;\sigma\;\in\;L^p(\Omega)\;\;\;\;\forall\;\,p\in\;[1,+\infty].$$
On the other hand, the transport density \,$\sigma$\, (see \eqref{transport density def}) is the total variation of a vector measure \,$v$\, solving the following problem (which is the {\it{continuous transportation problem}}\, proposed by Beckmann in \cite{Beck})
$$\mbox{(BP)} \qquad \min \left\{\int_{\Omega} \mathrm{d}|v|:\; v \in \mathcal{M}(\Omega,\mathbb{R}^d),\;\nabla\cdot v=\mu - \nu \right\},$$\\ 
where $\mathcal{M}(\Omega,\mathbb{R}^d)$ denotes the space of vector measures on $\Omega$. In fact, for a given optimal transport plan $\gamma$, let us define a vector measure \,$v^\gamma$\, and a scalar one \,$\sigma^\gamma$\, as follows 
\begin{equation}\label{vector measure}
 <v^\gamma,\psi>:=\int_{\Omega \times \Omega}\int_0^1 \psi(w_{x,y}(t)) \cdot {w}^\prime_{x,y}(t)\,\mathrm{d}t\,\mathrm{d}\gamma(x,y),\,\;\;\mbox{for all }\;\psi\in C(\Omega,\mathbb{R}^d),
 \end{equation}
and  
\begin{equation}\label{transport density}
 <\sigma^\gamma,\varphi>:=\int_{\Omega \times \Omega} \int_{0}^{1} \varphi(w_{x,y}(t)) |{w}^\prime_{x,y}(t)|\,\mathrm{d}t\,\mathrm{d}\gamma(x,y),\,\;\;\mbox{for all}\,\,\;\varphi\;\in\;C(\Omega),
 \end{equation}\\ 
where $w_{x,y}$ is a curve parameterizing the straight line segment connecting $x$ to $y$. Recalling \eqref{transport density def}, we observe easily that \,$\sigma^\gamma$\, is nothing but the transport density between \,$\mu$\, and \,$\nu$. It is not difficult to see that \,$v^\gamma=-\sigma^\gamma\,\nabla u$, where \,$u\,$ is a Kantorovich potential in the transportation of \,$\mu$\, onto \,$\nu$. In addition, one can show that the vector measure
\,$v^\gamma$\, is, in fact, a minimizer for (BP) and, one has
 $$ \min \mbox{(BP)}=\sigma^\gamma(\Omega)=\int_{\Omega \times \Omega} |x-y|\,\mathrm{d}\gamma=\min\mbox{(KP)}=\sup \mbox{(DP)}.$$\\ Moreover, every minimizer \,$v$\, for (BP) is of the form $v=v^\gamma$ \eqref{vector measure}, for some optimal transport plan $\gamma$ (see \cite{8}). This implies that if the source measure \,$\mu$\, or the target one \,$\nu$\, is absolutely continuous with respect to \,$\mathcal{L}^d$, then (BP)
  has a unique minimizer which is, by the way, in $L^1(\Omega,\mathbb{R}^d)$. So, this provides existence and uniqueness of the minimizer for the following minimal flow problem 
 \begin{equation} \label{minimal flow problem}
 \min \left\{\int_{\Omega}|v|\,\mathrm{d}x:\; v \in L^1(\Omega,\mathbb{R}^d),\;\nabla\cdot v= f,\,\,v \cdot n=0\,\,\,\mbox{on}\,\,\,\partial\Omega \right\}.
 \end{equation} \\  
 We recall that the unique minimizer $v$ of \eqref{minimal flow problem} belongs to $L^p(\Omega,\mathbb{R}^d)$ as soon as $f \in L^p(\Omega)$.
The goal of this paper is to generalize this result
  by showing the following novel one about the $L^{p,q}$ regularity of the transport density $\sigma$:
$$ f^+,\,f^-\;\in\;L^{p,q}(\Omega)\Rightarrow \;\sigma\;\in\;L^{p,q}(\Omega)$$
where $$L^{p,q}(\Omega)=\bigg\{f:\Omega \mapsto \mathbb{R}\,\,\,\mbox{measurable}\,:\,||t (\mathcal{L}^d(\{|f|\geq t\}))^{\frac{1}{p}}||_{L^q(\mathbb{R}^{+},\frac{\mathrm{d}t}{t})} < \infty\bigg\}.$$
\section{New estimates on the transport density}
In \cite{558,778,98}, the authors have already showed, using different techniques, the following $L^p$ summability on the transport density $\sigma$:
\begin{equation} \label{3.1}
f^\pm \in L^p(\Omega) \Rightarrow \sigma \in L^p(\Omega),\,\,\,\,\forall \,\,\,p\in [1,\infty].
\end{equation}
Here, our aim is to extend this result to the Lorentz space (see Appendix \ref{Appendix}). This means that we will prove the following implication
\begin{equation} \label{3.2}
f^\pm \in L^{p,q}(\Omega) \Rightarrow \sigma \in L^{p,q}(\Omega),\,\,\,\forall \,\,p \in (1,\infty),\,q \in [1,\infty].
\end{equation}
In this way, \eqref{3.1} becomes a particular case of \eqref{3.2}, when $q=p$. The strategy of the proof (which is already used in \cite{98}) is based on a displacement interpolation and an approximation by discrete measures. In all that follows at least one between \,$\mu$\, and \,$\nu$\, will be absolutely continuous with respect to \,$\mathcal{L}^d$. Then, there will exist an optimal transport map \,$T$\, for (MP) from \,$\mu$\, to \,$\nu$\, (or, from \,$\nu$\, to \,$\mu$) and one unique transport density \,$\sigma$\, associated to those measures (independent of the optimal transport plan $\gamma$). 
First, let us suppose that the target measure $\nu$ is finitely atomic and let us denote by $(x_i)_{i=1,...,n}$ its atoms, that is $$ \nu=\sum_{i=1}^n \alpha_i \,\delta_{x_i}.$$ 
\begin{proposition} \label{Prop 3.1}
Suppose that $\mu=f^+ \cdot \mathcal{L}^d$ with $f^+\in L^{p,q}(\Omega)$. If $\;p< d^\prime:=d/(d-1)$, then the unique transport density \,$\sigma$ associated with the transport of \,$\mu$ onto \,$\nu$ belongs to $L^{p,q}(\Omega)$.
\end{proposition}
\begin{proof}
Let $\gamma$ be an optimal transport plan from $\mu$ to $\nu$ and let $\sigma$ be the unique transport density between them.
Let $\mu_t$ be the standard interpolation between the two measures $\mu$ and $\nu$, that is $$\mu_t=(\Pi_t)_{\#}\gamma$$\\ where $\Pi_t(x,y)=(1-t)x + ty$. We see that \,$\mu_0=\mu$\, and \,$\mu_1=\nu$. Since the domain \,$\Omega$\, is bounded, it is evident, recalling \eqref{transport density}, that we have  $$ \sigma \leq C \int_0^1 \mu_t\,\mathrm{d}t. $$ 
Yet,
  $$||\sigma||_{L^{p,q}(\Omega)}^q=p \int_0^{+\infty}s^{q-1} (\mathcal{L}^d(\{x \in \Omega \,:\,\sigma(x)\geq s\}))^{\frac{q}{p}}\,\mathrm{d}s.$$
 As $$\mathcal{L}^d(\{x \in \Omega \,:\,\sigma(x) \geq s\}) \leq \mathcal{L}^d\bigg(\bigg\{x \in \Omega \,:\,
 \int_0^1 \mu_t(x)\,\mathrm{d}t\geq \frac{s}{C}\bigg\}\bigg),$$ 
 then it is easy to see that $$||\sigma||_{L^{p,q}(\Omega)} \leq C\;\bigg|\bigg|\int_0^1 \mu_t\,\mathrm{d}t\bigg|\bigg|_{L^{p,q}(\Omega)}.$$\\
Now, using Minkowski's inequality in the Lorentz space $L^{p,q}(\Omega)$, we get the following estimate \\$$||\sigma||_{L^{p,q}(\Omega)} \leq C_{p}\;\int_0^1||\mu_t||_{L^{p,q}(\Omega)}\;\mathrm{d}t.$$\\ 
 Since $\mu \ll \mathcal{L}^d$, then there exists an optimal transport map \,$T$\, from \,$\mu$\, to \,$\nu$. But \,$\nu$\, is finitely atomic, then \,$\Omega$\, is the disjoint union of a finite number of sets $\Omega_i=T^{-1}(\{x_i\})$. For each $i \in \{1,...,n\}$, set $$ \Omega_i(t):=(1-t)\Omega_i + tx_i,\,\,\,\,\forall\,\,\,t \in[0,1].$$\\
It is not difficult to see that \,$\Omega_i(t)$\, are essentially disjoint. Yet, $\mu_t$ is absolutely continuous and its density $f_t$ coincides on each set \,$\Omega_i(t)$\, with the density of a homothetic image of $f^+$, the homothetic ratio being $(1-t)$. This means that $f_t$ is concentrated on the union of \,$\Omega_i(t)$ and, for any $i \in \{1,..,n\}$, one has
  $$\qquad f_t(y)=\frac{f^+(\frac{y-tx_i}{1-t})}{(1-t)^d}\;\;\;\;\,\,\mbox{for a.e.}\;\;y \in \Omega_i(t).$$
For a fixed $s>0$, we have 
 $$\,\;\,\,\,\mathcal{L}^d(\{y\in\Omega\,:\,f_t(y) \geq s\}) $$
 
 $$=\sum_{i=1}^n \mathcal{L}^d\bigg(\bigg\{y\in\Omega_i(t) \,:\, \frac{f^+(\frac{y-tx_i}{1-t})}{(1-t)^d}\geq s\bigg\}\bigg)$$ 
 
 $$= \sum_{i=1}^n \mathcal{L}^d\bigg(\bigg\{y\in\Omega_i(t)\,:\, f^+\left(\frac{y-tx_i}{1-t}\right)\geq s(1-t)^d\bigg\}\bigg) $$
 
 $$=\sum_{i=1}^n (1-t)^d\,\mathcal{L}^d(\{x\in\Omega_i\,:\,f^+\left(x\right)\geq s(1-t)^d\})$$
 
 $$= (1-t)^d\,\mathcal{L}^d(\{x\in\Omega\,:\,f^+\left(x\right)\geq s(1-t)^d\}).$$\\
 Then, one has
 $$\int_0^{+\infty}s^{q-1}\,(\mathcal{L}^d(\{y \in \Omega\,:\,f_t(y) \geq s \}))^{\frac{q}{p}}\,\mathrm{d}s$$
 
 $$=\int_0^{+\infty}s^{q-1}(1-t)^{\frac{d q}{p}} (\mathcal{L}^d(\{x \in \Omega\,:\,f^+(x) \geq s(1-t)^d \}))^{\frac{q}{p}}\,\mathrm{d}s$$
 
 $$=\int_0^{+\infty}\bigg(\frac{s}{(1-t)^{d}}\bigg)^{q-1}(1-t)^{\frac{d q}{p}} \,(\mathcal{L}^d(\{x \in \Omega\,:\,f^+(x) \geq s \}))^{\frac{q}{p}}\frac{\mathrm{d}s}{(1-t)^d}$$
 
 $$=(1-t)^{-\frac{d q}{p^\prime}}\int_0^{+\infty}s^{q-1}\,(\mathcal{L}^d(\{x \in \Omega\,:\,f^+(x) \geq s \}))^{\frac{q}{p}}\mathrm{d}s,$$\\ 
where $\;p^\prime=p/(p-1)$. This implies that  $$||f_t||_{L^{p,q}(\Omega)} = (1-t)^{-\frac{d}{p^\prime}}\,||f^+||_{L^{p,q}(\Omega)}.$$
  Finally, we get $$||\sigma||_{L^{p,q}(\Omega)} \leq C_{p} \left(\int_0^1\;(1-t)^{-\frac{d}{p^\prime}}\mathrm{d}t\right) ||f^+||_{L^{p,q}(\Omega)}< +\infty. $$\\
Hence, $\sigma \in L^{p,q}(\Omega)$. $\qedhere$\\
\end{proof}
\begin{lemma} \label{Lemma 3.2}
Suppose that \,$\nu_n \rightharpoonup \nu$. If \,$\gamma_n\,$ is an optimal transport plan between \,$\mu$\, and \,$\nu_n$, then there exists a subsequence $(\gamma_{n_k})_{n_k}$ such that $\gamma_{n_k} \rightharpoonup \gamma $ and $\gamma \in \Pi(\mu,\nu)$ is an optimal transport plan between \,$\mu$ and \,$\nu$.
\end{lemma}
\begin{proof}
For each \,$n$, let \,$u_n$ be a Kantorovich potential between \,$\mu$\, and \,$\nu_n$\, such that \,$\min u_n=0$. Then, we see easily that there is a subsequence $(u_{n_k})_{n_k}$ such that $u_{n_k} \rightarrow u$ uniformly in $\Omega$. Yet, we have
$$\int_{\Omega \times \Omega} |x-y|\,\mathrm{d}\gamma_{n_k}(x,y)=\int_\Omega u_{n_k}\,\mathrm{d}(\mu - \nu_{n_k}).$$\\ 
Then, passing to the limit, we get\\
$$\int_{\Omega \times \Omega} |x-y|\,\mathrm{d}\gamma(x,y)=\int_\Omega u\,\mathrm{d}(\mu - \nu).$$\\
This implies that \,$\gamma$\, is an optimal transport plan between \,$\mu$\, and \,$\nu$, and \,$u$\, is the corresponding Kantorovich potential. $\qedhere$\\
\end{proof}
\begin{proposition} \label{Prop. 3.3}
If \,$\mu=f^+ \cdot \mathcal{L}^d$\, with \,$f^+\in L^{p,q}(\Omega)$ and \,$\nu$\, is any non-negative measure on $\Omega$, then, if $\;p< d^\prime:=d/(d-1)$, the unique transport density \,$\sigma$ associated with the transport of \,$\mu$ onto \,$\nu$ belongs to $L^{p,q}(\Omega)$.
\end{proposition}
\begin{proof}
Let us consider a regular grid $G^n\subset\Omega$ composed of approximately $C n^d$ points (take $G^n=\frac{1}{n}\mathbb{Z}^d \cap \Omega$) and let \,$P_n$ be the projection map from \,$\Omega$\, to \,$G^n$.
Set $$\nu_n:=(P_n)_{\#}\nu.$$ Then, $\nu_n$ is atomic with at most $C n^d$ atoms and $\nu_n \rightharpoonup  \nu$. 
Let $\sigma_n$ be the transport density associated with the transport of \,$\mu$\, onto \,$\nu_n$. By Proposition \ref{Prop 3.1}, we have that \,$\sigma_n \in L^{p,q}(\Omega)$\, and 
$$||\sigma_n||_{L^{p,q}(\Omega)} \leq C_{p} \left(\int_0^1\;(1-t)^{-\frac{d}{p^\prime}}\,\mathrm{d}t\right) ||f^+||_{L^{p,q}(\Omega)}. $$\\
This inequality, which is true in the discrete case, stays true at the limit as well, indeed, by Lemma \ref{Lemma 3.2}, $\sigma_{n} \rightharpoonup \sigma$, where \,$\sigma$\, is the unique transport density associated with the transport of $\mu$ onto $\nu$. But $(\sigma_{n})_{n}$ is bounded in $L^{p,q}(\Omega)$, then 
$$||\sigma||_{L^{p,q}(\Omega)} \leq \liminf_n ||\sigma_n||_{L^{p,q}(\Omega)} \leq C_{p} \left(\int_0^1\;(1-t)^{-\frac{d}{p^\prime}}\,\mathrm{d}t\right) ||f^+||_{L^{p,q}(\Omega)}< +\infty $$
 and $$\sigma \in L^{p,q}(\Omega). $$
 \end{proof}
 \begin{remark} \label{Remark}
 If we denote by \,$\mu_{n,t}$\, the interpolation between the two measures \,$\mu$\, and \,$\nu_n$, then $$\mu_{n,t} \rightharpoonup \mu_t.$$\\
Moreover, if \,$f_{n,t}$\, denotes the density of \,$\mu_{n,t}$, then 
$(f_{n,t})_{n}$ is bounded in $L^{p,q}(\Omega)$ and so, we have
$$||f_t||_{L^{p,q}(\Omega)} \leq \liminf_n ||f_{n,t}||_{L^{p,q}(\Omega)} =(1-t)^{-\frac{d}{p^\prime}} ||f^+||_{L^{p,q}(\Omega)}.$$ 
\end{remark}

 \noindent By Remark \ref{Remark}, we see that the measures \,$\mu_t$\, inherit some regularity ($L^{p,q}$ summability) from \,$\mu$\, exactly as it happens for homotheties of ratio $1-t$. This regularity degenerates as $t \rightarrow 1$, but we saw that this degeneracy produced no problem for $ L^{p,q}$ estimates on the transport density $\sigma$, provided $p<d/(d-1)$.
 Yet, for $p \geq d/(d-1)$, we need to exploit another strategy: suppose both $f^+$ and $f^-$ share some regularity assumption (belong to  $L^{p,q}$). Then we can give estimate on \,$f_t$\, for \,$t \to 0$\, starting from $f^+$ and for $t \to 1$ starting from $f^-$. This will avoid the degeneracy. \\ \\
 In fact, this strategy works but we must pay attention to one thing: in the previous estimates, $f_t$ is obtained as a limit from discrete approximations and so, it doesn't share a priori the same behavior of piecewise homotheties of $f^+$. And, when we pass to the limit, we do not know which optimal transport plan $\gamma$ will be selected as a limit of the optimal transport plans $\gamma_n$. This was not an issue in Proposition \ref{Prop. 3.3}, thanks to the uniqueness of the transport density $\sigma$ (since any optimal transport plan $\gamma$ induces the same transport density $\sigma$). But here we want to glue together estimates on \,$f_t$\, for \,$t \to 0$\, which have been obtained by approximating \,$f^-$\, and estimates on \,$f_t$\, for \,$t \to 1$\, which come from the approximation of $f^+$. Should the two approximations converge to two different transport plans, we could not put together the two estimates and deduce anything on $\sigma$. So, the lack of uniqueness of optimal transport plans may create a problem. Hence, the idea is to consider a strictly convex cost as $|x-y|^{1+\varepsilon}$, where $\varepsilon >0$, instead of $|x-y|$ since, in this case, the corresponding optimal transport plan $\gamma_\varepsilon$ will be unique (see, for instance, \cite{8,11}). We note that this strategy is different from the one given in \cite{98} where the author shows that the ``monotone optimal transport plan" can be
approximated in both directions.\\
 
\noindent Set $\mu_{\varepsilon,t}:=(\Pi_t)_{\#} \gamma_\varepsilon$. Then, one can prove as above that the density \,$f_{\varepsilon,t}$\, of \,$\mu_{\varepsilon,t}$\, satisfies the following estimate
 $$||f_{\varepsilon,t}||_{L^{p,q}(\Omega)} \leq (1-t)^{-\frac{d}{p^\prime}} ||f^+||_{L^{p,q}(\Omega)}.$$ \\
In the same way, one can also prove that
 $$||f_{\varepsilon,t}||_{L^{p,q}(\Omega)} \leq t^{-\frac{d}{p^\prime}} ||f^-||_{L^{p,q}(\Omega)}.$$
Consequently, $$||f_{\varepsilon,t}||_{L^{p,q}(\Omega)} \leq \min\{(1-t)^{-\frac{d}{p^\prime}}||f^+||_{L^{p,q}(\Omega)},\; t^{-\frac{d}{p^\prime}}||f^-||_{L^{p,q}(\Omega)}\} $$
 $$ \,\,\,\,\,\,\leq  C\,\max\{ ||f^+||_{L^{p,q}(\Omega)},\;||f^-||_{L^{p,q}(\Omega)}\}.$$
\\ 
 Finally, we get the following:
 \begin{proposition}
 Suppose that $\mu=f^+ \cdot \mathcal{L}^d$ and \,$\nu=f^- \cdot \mathcal{L}^d$ with $f^\pm \in L^{p,q}(\Omega)$ and let $\sigma$ be the unique transport density associated with the transport of \,$\mu$ onto $\nu$. Then, $\sigma$ belongs to $L^{p,q}(\Omega)$ as well.
 \end{proposition}
 \begin{proof}
 Let $\gamma_{\varepsilon}$ be the unique optimal transport plan in the transportation of $\mu$ onto $\nu$ with transport cost $|x-y|^{1+\varepsilon}$. Then, we see easily that
  \,$\gamma_{\varepsilon} \rightharpoonup \gamma$, where \,$\gamma$\, is an optimal transport plan in the transportation of \,$\mu$\, onto \,$\nu$\, with transport cost $|x-y|$. 
This implies that $\mu_{\varepsilon,t} \rightharpoonup \mu_t$. Yet, $(f_{\varepsilon,t})_{\varepsilon}$ is bounded in $L^{p,q}(\Omega)$. Then,
 $$||f_t||_{L^{p,q}(\Omega)} \leq \liminf\limits_{\varepsilon}||f_{\varepsilon,t}||_{L^{p,q}(\Omega)}\leq C\,\max\{ ||f^+||_{L^{p,q}(\Omega)},\;||f^-||_{L^{p,q}(\Omega)}\},\,\,\,\,\,\forall\,\,\,t \in [0,1].$$\\
Consequently, $\sigma \in L^{p,q}(\Omega). \qedhere $
 \end{proof}
\section{Appendix: The Lorentz space} \label{Appendix}
\begin{definition}
 The Lorentz space on \,$\Omega$ is the space of measurable functions \,$f$ on \,$\Omega$ such that the following quasinorm is finite
$$ ||f||_{L^{p,q}(\Omega)}=p^{\frac{1}{q}}||t (\mathcal{L}^d(\{|f|\geq t\}))^{\frac{1}{p}}||_{L^q(\mathbb{R}^{+},\frac{\mathrm{d}t}{t})}$$
 where \,$0<p<\infty$\, and \,\,$0<q\leq\infty$. Thus, when $q<\infty$, $$||f||_{L^{p,q}(\Omega)}=p^{\frac{1}{q}}\left(\int_0^{+\infty}t^{q-1} (\mathcal{L}^d(\{|f|\geq t\}))^{\frac{q}{p}}\,\mathrm{d}t\right)^{\frac{1}{q}}$$
and, when $q=\infty$,
$$||f||_{L^{p,\infty}(\Omega)}^p=\sup_{t>0}\left(t^p \mathcal{L}^d(\{x:|f(x)|\geq t\})\right).$$
\end{definition}
\begin{remark}
 It is not difficult to observe that $L^{p,1}(\Omega) \subset L^{p,p}(\Omega)=L^p(\Omega) \subset L^{p,\infty}(\Omega)$, which means that the Lorentz spaces $L^{p,q}$ are generalisations of the $L^p$ spaces. 
 \end{remark}
On the other hand, the quasinorm is invariant under rearranging the values of the function $f$, essentially by definition. In particular, given a measurable function $f$ defined on $\Omega$, its decreasing rearrangement function $f^*:\mathbb{R}^+ \mapsto \mathbb{R}^+$ can be defined as\\
$$ f^*(t)=\inf\{ \alpha\;\in\;\mathbb{R}^+ \,:\, \mathcal{L}^d(\{x\;\in\;\Omega\,:\,|f(x)|> \alpha\}) \leq t\},$$\\
where, for notational convenience, inf $\emptyset$ is defined to be $+\infty$. 
It is easy to see that the two functions $|f|$ and $f^*$ are equimeasurable, meaning that \\$$\mathcal{L}^d(\{x\;\in\;\Omega\,:\,|f(x)|> \alpha\})=|\{t>0:f^*(t)>\alpha\}|,\;\;\;\forall\,\;\alpha>0. $$\\
So, the Lorentz quasinorms are given by 
$$||f||_{L^{p,q}}=\left(\int_0^{+\infty}\bigg(t^{\frac{1}{p}} f^*(t)\bigg)^{q}\frac{\mathrm{d}t}{t}\right)^{\frac{1}{q}} \;\;\mbox{when}\;\;q<\infty$$
and
$$||f||_{L^{p,\infty}}=\sup_{t>0}\; t^{\frac{1}{p}}f^*(t).$$
Set $$f^{**}(t)=\frac{1}{t}\int_0^t f^*(s) \,\mathrm{d}s.$$
We define
 $$ ||f||_{p,q}=\left(\int_0^{+\infty}\bigg(t^{\frac{1}{p}} f^{**}(t)\bigg)^{q}\frac{\mathrm{d}t}{t}\right)^{\frac{1}{q}} \;\;\mbox{when}\;\;q<\infty$$
and
$$||f||_{p,\infty}=\sup_{t>0}\; t^{\frac{1}{p}}f^{**}(t).$$
\begin{theorem}
If $1<p<\infty,\,1\leq q \leq \infty$, then $||.||_{p,q}$ is a norm on $L^{p,q}$ and hence, $(L^{p,q},||.||_{p,q})$ is a normed space. More precisely, 
$$ ||f||_{L^{p,q}} \leq ||f||_{p,q} \leq \frac{p}{p-1}||f||_{L^{p,q}}.$$
This means that the quasinorms $||.||_{L^{p,q}}$ and \,$||.||_{p,q}$ are equivalent. Moreover, $L^{p,q}$ is a Banach space and, the dual of $L^{p,q}$ is isomorphic to $L^{p^\prime,q^\prime}$, where \,$\frac{1}{p} + \frac{1}{p^\prime}=1$ and \,$\frac{1}{q} + \frac{1}{q^\prime}=1$
($L^{p,q}$ is also reflexive for \,$1<p\,,\,q<\infty$).
\end{theorem}
\begin{proof}
See, for instance, \cite{Lorentz}. $\qedhere$
\end{proof}
{\bf{Acknowledgments:}}
the author would like to thank Prof. Filippo Santambrogio for interesting
suggestions and comments.
  
\end{document}